\documentclass[11pt]{article}

\usepackage{epsfig,amsmath,amsfonts,amssymb,latexsym,color,amsthm,hhline,multirow,subfigure,graphicx}
\usepackage{amsmath}
\usepackage{amsthm}
\usepackage{enumerate}
\usepackage{amsfonts}
\usepackage{verbatim}
\usepackage{graphicx}
\usepackage{float}
\usepackage{epstopdf}
\usepackage{amssymb}
\usepackage{bm}
\usepackage{bbm}
\usepackage{hyperref}
\usepackage{color}
\usepackage{curves}
\usepackage{tikz}
\usepackage[hang]{caption}
\usepackage{mathrsfs}

\newtheorem{theorem}{Theorem}[section]
\newtheorem{corollary}{Corollary}[section]
\newtheorem{lemma}{Lemma}[section]
\newtheorem{proposition}{Proposition}[section]
\newtheorem{definition}{Definition}

\newcommand{\Z}{{\mathbb Z}}
\newcommand {\PP}{{\mathbb P}}

\newcommand{\N}{\mathbb{N}}

\newcommand\E{\mathbb{E}\,}
\newcommand\Prob{\mathbb{P}}
\renewcommand\epsilon{\varepsilon}
\renewcommand\phi{\varphi}

\begin{document}

\title{The Schelling model on $\Z$}
\author{Maria Deijfen\thanks{Department of Mathematics, Stockholm University; {\tt mia@math.su.se}} \and  Timo Hirscher \thanks{Department of Mathematics, Stockholm University; {\tt timo@math.su.se}\newline\indent\,
This work was in part supported by grant 2014-4948 from the Swedish Research Council.
\newline\indent\, MD is grateful to Jessica Dj\"{a}kneg\aa rd for introducing her to the Schelling model during the work on \cite{Jessica}.}}
\date{June 2019}

\maketitle

\begin{abstract}
\noindent A version of the Schelling model on $\Z$ is defined, where two types of agents are allocated on the sites. An agent prefers to be surrounded by other agents of its own type, and may choose to move if this is not the case. It then sends a request to an agent of opposite type chosen according to some given moving distribution and, if the move is beneficial for both agents, they swap location. We show that certain choices in the dynamics are crucial for the properties of the model. In particular, the model exhibits different asymptotic behavior depending on whether the moving distribution has bounded or unbounded support. Furthermore, the behavior changes if the agents are lazy in the sense that they only swap location if this strictly improves their situation. Generalizations to a version that includes multiple types are discussed. The work provides a rigorous analysis of so called Kawasaki dynamics on an infinite structure with local interactions.
\end{abstract}

\section{Introduction}

The Schelling model of segregation was formulated by Thomas Schelling in the 1960's as an attempt to explain the occurrence of racial segregation in terms of individual preferences rather than policies of central authorities; see \cite{SchI,SchII}. It is one of the first examples of so-called agent based modeling in the economic literature, where the dynamics is formulated in terms of actions of autonomous agents and the interest then concerns large scale properties of such populations. In probability, this type of models is studied in the area of interacting particle systems. In the work of Schelling, two types are distributed on a finite part of $\mathbb{Z}^d$ ($d=1,2$) that also contains a certain percentage of empty sites and, depending on the number of agents of opposite type in some neighborhood around it, an agent can choose to move to another (empty) location where the neighborhood contains fewer agents of the opposite type. Specifically, if the fraction of agents of opposite type in the neighborhood exceeds a certain threshold, then the agent is dissatisfied and therefore prone to move. One of the main lessons learnt from the model is that already a mild preference for being surrounded by alike agents can lead to massive segregation on a macroscopic level.

The Schelling model has been widely studied in the economic literature, mainly by aid of simulations, but so far there is not much rigorous work. A number of choices have to be made to specify the precise dynamics of the model, e.g.\ the proportions of the types and empty sites, the choice of the new location of a moving agent and whether moves that do not strictly improve the neighborhood composition for an agent are permitted. We study a version of the model on the line $\Z$, and show that the model exhibits qualitatively different asymptotic behavior depending on some of these choices. In contrast to most existing rigorous work, we work with the original dynamics where the proportions of the types are fixed, that is, the agents keep their type but change their location (so-called Kawasaki dynamics). This dynamics is novel compared to related models in statistical mechanics where particles instead switch type depending on their neighborhood (so-called Glauber dynamics); see e.g.\ \cite{scaling} for recent work on a Schelling type model. Our results concern a version of the model that does not include empty sites, but that is instead driven by pairwise swaps. In other work on the Schelling model with swaps, the distance that a dissatisfied agent is willing to move typically grows to infinity with the size of the system; see e.g. \cite{Young, BarmI,Brandt}. In our model, despite the system being infinite, swaps occur based on a given moving distribution which is a.s.\ finite, so that the rules for interactions are in this sense local.

Our model works roughly as follows. Each site of $\Z$ is initially occupied by either a type 1 or a type 2 agent. An agent is satisfied if both its nearest neighbors are of the same type (that is, if the proportion of agents of the same type among its nearest neighbors is strictly larger than 1/2). An agent that is not satisfied makes attempts to move at rate 1. It then sends a request to an agent chosen according to some moving distribution. If the involved agents are not worse off after the move -- that is, if they do not have strictly fewer nearest neighbors of the same type at the new locations -- then they switch places. We show that the asymptotic behavior is qualitatively different depending on whether the moving distribution has bounded or unbounded support. When the support is unbounded, then the configuration is homogenized in the sense that it will locally consist of only one type, but the type changes infinitely often. This behavior resembles that of the voter model. When the support is bounded, however, the configuration on any finite section of $\Z$ freezes after finite time and simulations indicate that it then consists of intervals only slightly exceeding the maximal moving horizon. Furthermore, if the agents are lazy in the sense that they are not willing to move if the move does not strictly improve their neighborhood composition, then the behavior of the model is different and it freezes regardless of the moving distribution.

\subsection{Definition of the model}

As indicated above, our main results concern a model on $\Z$ with two types. However, we formulate the model slightly more generally on $\Z^d$ including $c\geq 2$ types. To this end, let $\bm{\eta}(0)=\big\{\eta_v(0)\big\}_{v\in\Z^d}$ be a collection of i.i.d.\ random variables that constitute the initial configuration. Their support is $\{1,\dots,c\}$ for some $c\in\N$, where $\eta_v(0)=i$ corresponds to the site $v$ being occupied by an agent of type $i$ at time $t=0$. Furthermore, let $\mu$ denote a distribution on $\Z^d$ with the following property: There exists an $h\in\N_0\cup\{\infty\}$, which we will call the {\em moving horizon}, such that
\begin{equation}\label{horizon}
\mu(u)>0 \quad\text{if and only if}\quad \lVert u \rVert_1\leq h.
\end{equation}
The distribution $\mu$ will control where an agent that is not satisfied at its current position (see below) attempts to move and \eqref{horizon} stipulates that it puts strictly positive mass on all sites within $L^1$-distance $h$ from the agent itself, where $h$ is the largest possible distance for a move.

For all $v\in\Z$, let $N_v$ be a homogeneous Poisson counting process with rate $1$ and $\boldsymbol{Q}_v=\big(Q_{v,n}\big)_{n\in\N}$ an i.i.d.\ sequence of random variables with marginal distribution $\mu$. These are taken such that we have independence both between $\eta_v(0),\ N_v$ and $\boldsymbol{Q}_v$  as well as between any two of these random entities associated with different sites. Write $N(v)$ for the set of nearest neighbor sites of $v$ and introduce the satisfaction function
\begin{equation}\label{satisfunction}
\mathrm{sat}_v(t)=\big|\big\{u\in N(v);\;\eta_u(t)=\eta_v(t)\big\}\big|-\big|\big\{u\in N(v);\;\eta_u(t)\neq\eta_v(t)\big\}\big|.
\end{equation}
An agent is dissatisfied at time $t$ if $\mathrm{sat}_v(t)\leq 0$, that is, if it has at least as many neighbors of different type as of its own type.

Our variant of the Schelling model now evolves according to the following rules. When the Poisson process $N_v$ has its $n$th jump and $v$ is currently occupied by a dissatisfied agent, the latter turns {\em active} and proceeds as follows. First, site $v+Q_{v,n}$ is evaluated in terms of the satisfaction of the agent (currently placed at $v$) after moving to this site. Then, if this is at least as high as in its current position $v$ and the agent placed there is dissatisfied as well and her satisfaction does not decrease by switching with the agent at $v$, then the two agents switch places. Note that, with $c=2$, the agents always agree on whether to swap or not in the sense that, if the satisfaction of the agent currently placed at $v$ is at least as high at $v+Q_{v,n}$, then the satisfaction of the agent currently placed at $v+Q_{v,n}$ is also at least as high at $v$. If $v$ chooses a site that is occupied by an agent of its own type, then there is effectively no change in the configuration and we can consider the swap to be denied. Observe that the memorylessness of the Poisson processes ensures that the model as defined above is a Markov process.

\subsection{Results}\label{ss:res}

Our main results concern the above Schelling model in $d=1$ with two types. When $h=\infty$, so that moves over arbitrarily long distances are possible, then the long-term behavior of the model is reminiscent of the asymptotics of the well-known voter model.

\begin{theorem}\label{d=1,unbounded}
	Consider the Schelling model on $\Z$ with two types and $h=\infty$.
	\begin{enumerate}[(a)]
	\item For any two sites $u,v\in\Z$, we have that
	\begin{equation}\label{limit}
	\lim_{t\to\infty}\Prob\big(\eta_u(t)\neq\eta_v(t)\big)=0.
	\end{equation}
	This is equivalent to $\big|\eta_u(t)-\eta_v(t)\big|\stackrel{L^1}{\longrightarrow}0$, as $t\to\infty$.
	\item Almost surely both types of agents occupy a given site $v\in\Z$ at arbitrarily large times, that is,
	\begin{equation}
	\lim_{t\to\infty}\eta_v(t) \text{ a.s.\ does not exist.}
	\end{equation}
	This entails that almost surely $\big|\eta_u(t)-\eta_v(t)\big|$ does not converge.
	\end{enumerate}
\end{theorem}

When $h\in\N$, the asymptotics is qualitatively different. Specifically, the configuration almost surely fixates in the sense that $\eta_v(t)$ converges as $t\to\infty$. This is essentially due to the fact that the positions of monochromatic sections of length strictly larger than $h$ are stable in the sense that there cannot be any swaps of agents placed on different ends/sides of the section.

\begin{theorem}\label{d=1,bounded}
	In the Schelling model on $\Z$ with two types and $h\in\N$, the configuration almost surely converges, that is, for all $v\in \Z$
		\[\lim_{t\to\infty}\eta_v(t) \text{ a.s.\ exists.}\]
\end{theorem}
Write $p_i=\PP(\eta_v(0)=i)$ for $i=1,2$. It is straightforward to verify that $\Prob\big(\lim\limits_{t\to\infty}\eta_v(t)=1\big)=p_1$ and $\Prob\big(\lim\limits_{t\to\infty}\eta_v(t)=2\big)=p_2$; see Lemma \ref{probconst}.

Now consider a modification of the model where agents only move if they strictly improve their satisfaction function. We refer to this as the agents being lazy. It turns out that this changes things drastically in that the configuration then always converges, regardless of whether $h$ is finite or infinite. The essential difference to the original dynamics is that, with lazy agents, every swap must involve at least one singleton agent -- that is, an agent currently having no neighbor of her own color -- that does not remain singleton after the swap. Certainly, singletons can and will be created by the dynamics, but this turns out not to play a crucial role in the long run.

\begin{proposition}\label{prop:lazy}
	Consider the two-type Schelling model on $\Z$, with lazy agents. Then, for all $v\in \Z$:
	\[\lim_{t\to\infty}\eta_v(t) \text{ a.s.\ exists.}\]
\end{proposition}

Another variation of the model is to lower the satisfaction threshold so that an agent is dissatisfied at time $t$ if $\mathrm{sat}_v(t)\leq \tau<0$. This means that only agents with no neighbor of their own color will be dissatisfied, and will also eliminate the non-fixation in the unbounded case (irrespective of whether the agents are lazy or not).

\begin{proposition}\label{prop:tau<0}
	Consider the Schelling model on $\Z$, with two types and dissatisfaction threshold $\tau<0$. Then, for all $v\in \Z$:
	\[\lim_{t\to\infty}\eta_v(t) \text{ a.s.\ exists.}\]
\end{proposition}

The rest of the paper is organized as follows. Extensions and possible generalizations of the model and results are discussed in Section \ref{ss:open}. In particular, we describe to what extent our results extend to the case with more than two types of agents. Section \ref{ss:ref} contains some further references to work on Schelling type models. After a preliminary section, Theorems \ref{d=1,unbounded} and \ref{d=1,bounded} are then proved in Sections \ref{sec:unbounded} and \ref{sec:bounded}, respectively, while the short proof of Propositions \ref{prop:lazy} and \ref{prop:tau<0} is given in Section \ref{lazy}. We then establish some extensions to multiple types in Section \ref{sec:ex}.

\subsection{Extensions and further work}\label{ss:open}

Here we describe some possibilities for further work on our version of the Schelling model.\smallskip

\noindent\textbf{Multiple types.} Our model allows for an arbitrary number $c$ of agent types. For a given agent, alike types are counted with +1 in the satisfaction function, while all other types are counted with -1. Theorem \ref{d=1,unbounded} extends to the case with $c\geq 3$ types of agents; see Section \ref{sec:ex}. When the moving horizon is bounded, there are configurations that do not stabilize in a finite system; see Section \ref{sec:ex}.  We do not know if these will persist in the long run on $\Z$, and leave this for future work.\smallskip

\noindent\textbf{Empty sites.} The original version of the Schelling model includes vacant sites and moves occur only in that an agent moves to a preferable vacant site. Our model could also be modified to include empty sites. It then has to be specified if both swaps and moves to empty sites are possible (or only moves to empty sites), and what effect empty sites have in the satisfaction function. As for the satisfaction function, one possibility is to let empty sites contribute 0, meaning that agents are indifferent to empty sites in their neighborhood. Some of our arguments apply also in this setting with both swaps and moves to empty sites, but we leave a systematic analysis of the effect of empty sites for future work.\smallskip

\noindent\textbf{Higher dimensions.} It would be natural to analyze the model also in dimensions $d\geq 2$. Schelling for instance worked in $d=1$ and $d=2$. Our methods do not immediately apply for $d\geq 2$, but we would expect the behavior of the model with two types (and no empty sites) to be qualitatively similar at least for $d=2,3$.\smallskip

\noindent\textbf{The neighborhood.} In the present formulation of the model, the satisfaction function is based on the nearest neighbors in the $L^1$-metric, that is, the neighborhood that is considered when an agent evaluates a position consists of the neighboring sites in the graph. One could of course also look at other neighborhoods. Schelling for instance based his analysis on Moore neighborhoods, consisting of all sites in a unit cube centered at the agent (in $d=1$ this is of course equivalent to the nearest neighbors). It would be natural to consider also larger neighborhoods, consisting e.g.\ of sites at $L^1$-distance at most $n$ for some $n\geq 2$. Refined versions of the satisfaction function could then be introduced, where the influence of a given agent in a neighborhood is weighted by its distance to the central agent.

\subsection{Related work}\label{ss:ref}

We do not intend to survey the vast amount of work on Schelling-type models in the economic literature, but content ourselves with mentioning some previous work in the mathematical literature. One of the first examples is \cite{Pollicott}, where a finite size system is analyzed consisting of agents that are all prepared to move if this improve their situation -- there is hence no threshold for dissatisfaction as in the original model. Moves then occur to locations anywhere in the system, meaning that the dynamics is not local as in our model, and the results amount to a characterization of stable configurations in terms of minimizers for a certain variational problem.

Models based on pairwise swaps (rather than moves to empty sites) were first studied in \cite{Young}. In \cite{Young} also so-called perturbed Schelling dynamics was introduced, referring to models where agents act against their preference with some small probability. Perturbed models then dominated the field for a number of years, see e.g.\ \cite{Zhang}. We also mention the more recent work \cite{Randall}, where the mixing time of a two-dimensional perturbed model is analyzed. Unperturbed one-dimensional models based on pairwise swaps are studied in \cite{BarmI,Brandt}. In contrast to our model however, the size of the region in which swaps occur grows to infinity with the size of the system, implying that the dynamics is not asymptotically local. Similar models in two and three dimensions are analyzed in \cite{BarmIII}. Another model class, inspired by the Schelling dynamics, consists of processes where agents do not move, but shift type depending on the composition of the their neighborhood. Work on this model class include e.g.\ \cite{BarmII,scaling,Immor}.

\section{Preliminaries}

In this section we introduce basic tools that will be used in the analysis and derive some elementary results for our Schelling model.

\subsection{Moving agents, mass transport and ergodicity}

To start with, let $\bm{\eta}(0)$, $\{N_v\}_{v\in\Z^d}$ and $\{\boldsymbol{Q}_v\}_{v\in\Z^d}$ be as defined above. Further, let $Y_v$ stand for the triple consisting of $\eta_v(0)$, $N_v$ and $\boldsymbol{Q}_v$. Observe that this makes $\boldsymbol{Y}=(Y_v)_{v\in\Z^d}$ a set of i.i.d.\ random elements, which embodies the entire randomness of the model.

As a first tool, we introduce the so-called {\em mass transport principle}. To fit our purposes, let a {\em mass transport} be a random function $m:\ \Z^d \times \Z^d\to [0,\infty)$, that is invariant in law under translations of $\Z^d$.
For subsets $A,B\subseteq \Z^d$, the value \[m(A,B):=\sum_{u\in A, v\in B} m(u,v)\] should be interpreted as the total mass sent from sites in $A$ to sites in $B$.

\begin{lemma}[Mass-transport principle]\label{MTP}
	Let $m$ be a mass transport and $v\in\Z^d$. Then
	\begin{equation}
	\E\big[m(v,\Z^d)\big] = \E\big[m(\Z^d,v)\big].
	\end{equation}
\end{lemma}

\begin{proof}
	As $m$ is non-negative and its law translation-invariant, it holds that
	\[	\E\big[m(v,\Z^d)\big]= \sum_{u\in\Z^d}\E\big[m(v,u)\big]= \sum_{u\in\Z^d}\E\big[m(2v-u,v)\big]= \E\big[m(\Z^d,v)\big].\vspace*{-1em}\]
\end{proof}

For a more general version of the mass transport principle, we refer to \cite[pp.\ 43]{BLPS}. Note that $\boldsymbol{Y}$ being translation invariant in law entails that any non-negative function $m$, which is a factor of $\boldsymbol{Y}$, necessarily constitutes a mass transport. The above lemma allows us to readily derive statements like the following:
\begin{lemma}\label{gentle}
	Consider the Schelling model on $\Z^d$ as described above and let $t\geq0,\ \epsilon>0$. The probability that no agent evaluates
	site $v\in \Z^d$ for a potential move during the time interval $(t,t+\epsilon]$ is bounded from below by $1-\epsilon$.
\end{lemma}

\begin{proof}
	Consider the mass transport $m(v,u;\boldsymbol{Y})$ given by the number of times during $(t,t+\epsilon]$, at which the Poisson process $N_v$ associated with site the $v$ has a jump, and there currently is a dissatisfied agent located at $v$, which then chooses to evaluate site $u$ for a potential move. Using Lemma \ref{MTP} and the fact that $\E[N_v(t+\epsilon)-N_v(t)]=\epsilon$, we get that
	\begin{equation}\label{mtpineq}
	\epsilon\geq\E\big[ m(v,\Z^d)\big]=\E \big[m(\Z^d,v)\big].
	\end{equation}
	Let $X$ denote the number of times an agent evaluates site $v$ during  $(t,t+\epsilon]$. The last expression in \eqref{mtpineq} equals $\E[X]$, and we obtain that
	$$\Prob(X\geq 1)\leq \E(X)\leq\epsilon.\vspace*{-2em}$$
\end{proof}

By the same argument, using a standard union bound, we immediately get a slight generalization to any finite set of vertices.
\begin{corollary}\label{nopass}
	Consider the Schelling model on $\Z^d$, let $t\geq0,\ \epsilon>0$ and $S\subseteq Z^d$ be a finite
	set of sites. Then, the probability that no agent placed outside $S$ evaluates a site inside $S$ for a potential
	move during the time interval $(t,t+\epsilon]$ is bounded from below by $1-\epsilon|S|$.
\end{corollary}

Also the probabilities of a site to be of a certain type are easily seen to stay constant over time. Let $\boldsymbol{p}:=(p_1,\dots,p_c)$ be the probability mass function for the initial state of a given vertex $v$, where $p_i$ is the probability that $v$ is occupied by an agent of type $i$ at time $t=0$.

\begin{lemma}\label{probconst}
	Consider the Schelling model on $\Z^d$ with $c$ agent types and let $v\in\Z^d$. For any $t\geq0,$ we have that
	\[\Prob\big(\eta_v(t)=1,\dots,\eta_v(t)=c\big)=\boldsymbol{p}.\]
\end{lemma}

\begin{proof}
	Fix $v\in\Z^d,\ t>0$ and define two mass transports: As in the proof
	of Lemma \ref{gentle}, let $m(u,v)$ count the number of times an agent placed at $u$ evaluates site $v$ for a potential move, here however
	during the time period $[0,t]$. Furthermore, let $m_i(u,v)$ be the number of times during $[0,t]$ an agent of type $i\in\{1,\dots,c\}$ moves
	from site $u$ to $v$. A swap of two agents, placed at $u$ and $v$ respectively, can be initiated by either of the two, implying that
	\begin{equation}\label{dominated}
	m_i(u,v)\leq m(u,v)+m(v,u).
	\end{equation}
	Trivially, $\E[m(v,\Z^d)]\leq\E[N_v(t)]=t$. Hence, \eqref{dominated} combined with the mass transport principle (Lemma \ref{MTP}) entails that
	\begin{equation}\label{bdd}
	\E\big[m_i(v,\Z^d)\big]=\E\big[m_i(\Z^d,v)\big]\leq 2t.
	\end{equation}
	Now observe that
	$\mathbbm{1}_{\{\eta_v(0)=i\}}+m_i(\Z^d,v)=m_i(v,\Z^d)+\mathbbm{1}_{\{\eta_v(t)=i\}}$.
	Taking expectations and using \eqref{bdd} yields $p_i=\Prob\big(\eta_v(t)=i)$, as desired.
\end{proof}

Next, we provide a lemma that will be useful in dimension $d=1$ and is in essence Birkhoff's pointwise ergodic theorem. However, in contrast to its standard formulation (cf.\ for instance \cite[Theorem 7.2.1]{Durrett}), we also want to allow sequences of {\em random} nested finite subsets.
Let $T$ denote the shift to the left on $\Z$, i.e.\ $T(v)=v-1$. Given a two-sided sequence $\boldsymbol{X}=(X_v)_{v\in\Z}$, write
$T\boldsymbol{X}$ for the sequence in which all labels are shifted down by one, i.e.\ the value at $v$ is taken to be $X_{v+1}$ for all $v$.

For $d=1$, the above defined $\boldsymbol{Y}=(Y_v)_{v\in\Z}$ is an i.i.d.\ sequence. Hence, ergodicity implies that, for any integrable function $f$
of $\boldsymbol{Y}$, averages of $f$ values attributed to shifts of $\boldsymbol{Y}$ converge to the mean:

\begin{lemma}\label{ergodic}
	Let $\boldsymbol{Y}=(Y_v)_{v\in\Z}$ be as above and $f$ be a real-valued integrable function of $\boldsymbol{Y}$. Further, let
	$(S_n)_{n\in\N}$ be a nested (possibly random) sequence of finite sections of $\Z$, which are strictly increasing in size. Then
	\begin{equation}\label{erglim}
	\lim_{n\to\infty}\frac{1}{|S_n|}\,\sum_{k\in S_n} f(T^k\boldsymbol{Y})= \E\big[f(\boldsymbol{Y})\big]\quad\text{a.s.}
	\end{equation}
\end{lemma}

\begin{proof}
	Bearing in mind that any integrable factor of an i.i.d.\ sequence is ergodic (with respect to $T$, see for instance \cite[Theorem 7.1.3]{Durrett}), the well-known pointwise ergodic theorem of Birkhoff (cf.\cite[Theorem 7.2.1]{Durrett}) implies that
	\begin{equation}\label{ergsimp}
	\lim_{n\to\infty}\frac{1}{n}\,\sum_{k=1}^n f(T^k\boldsymbol{Y})= \E\big(f(\boldsymbol{Y})\big)\quad\text{a.s.}
	\end{equation}
    By translation invariance, we can assume that $0\in S_1$ without loss of generality. Let $S_n=\{-U_n,\dots,V_n\}$, where $(U_n)_{n\in\N}$ and
	$(V_n)_{n\in\N}$ are (a.s.) non-decreasing sequences of $\N_0$-valued random variables such that almost surely $\lim_{n\to\infty}(U_n+V_n)=\infty$.
	
	If the sequence of sections grows one-sidedly, e.g.\ $\lim_{n\to\infty}U_n=u\in\N_0$ and $\lim_{n\to\infty}V_n=\infty$, the claim
	follows directly from \eqref{ergsimp}. If both $(U_n)_{n\in\N}$ and $(V_n)_{n\in\N}$ are unbounded, then $U_n\geq 1$ holds for
	$n$ large enough and we can write
	\[\frac{1}{|S_n|}\,\sum_{k\in S_n} f(T^k\boldsymbol{Y})=\frac{U_n}{|S_n|}\cdot\frac{1}{U_n}\,\sum_{k=1}^{U_n} f(T^{-k}\boldsymbol{Y})
	+\frac{V_n+1}{|S_n|}\cdot\frac{1}{V_n+1}\,\sum_{k=0}^{V_n} f(T^k\boldsymbol{Y}).\]
	
	\noindent Since, conditioned on the event $\{\lim_{n\to\infty}U_n=\lim_{n\to\infty}V_n=\infty\}$, the two random variables
	$\frac{1}{U_n}\,\sum_{k=1}^{U_n} f(T^{-k}\boldsymbol{Y})$ and $\frac{1}{V_n+1}\,\sum_{k=0}^{V_n} f(T^k\boldsymbol{Y})$
	both converge almost surely to $\E\big[f(\boldsymbol{Y})\big]$ by \eqref{ergsimp}, the claimed convergence of the left hand side --
	being a (random) convex combination of these two -- follows.\vspace*{-1em}
\end{proof}

\smallskip
\subsection{Separators and segregation}\label{ss:sep}

In the main part of the paper, we will focus on the one-dimensional model with two types of agents, say blue and red, i.e.\ $c=2$ and $p_1=1-p_2\in(0,1)$. In this setting, a configuration is determined by the color of one single vertex plus the locations of all nearest neighbor edges with endpoints occupied by agents of a different type. We will call these edges {\em separators}, as they represent boundaries which separate colors (cf.\ Figure \ref{separator}), and write $\mathcal{S}(t)$ for the (random) set of locations of separators at a given time $t$, that is,
\begin{equation}\label{S(t)}
\mathcal{S}(t)=\{\langle v,v+1\rangle;\;\eta_v(t)\neq\eta_{v+1}(t)\}.
\end{equation}
\begin{figure}[H]
	\centering
	\includegraphics[scale=0.86]{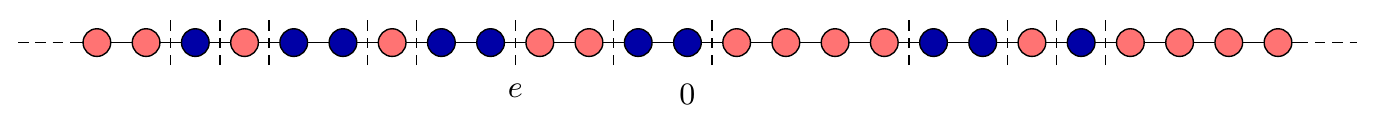}
	\caption{In the two-color the configuration is described by the separators.\label{separator}}
\end{figure}
Note that, since agents never impair their number of alike neighbors when moving, no new separators are ever created, but the dynamics moves existing separators in simple random walks on the edges of $\Z$ until they coalesce. Write $p(t)$ for the probability that a given edge is a separator at time $t\geq0$.

\begin{lemma}\label{decrease}
	Consider the Schelling model on $\Z$ with two types. Then, $t\mapsto p(t)$ is non-increasing in $t\geq0$.
\end{lemma}

\begin{proof}
	Taking $S_n:=\{-n,\dots,n-1\},\ n\in \N,$ as sequence of nested increasing sets and
	$f(\boldsymbol{Y})=\mathbbm{1}_{\{\langle 0,1\rangle \text{ is a separator at time }t\}}$, it follows from Lemma
	\ref{ergodic} that the probability that an edge is a separator at time $t$ almost surely equals the spatial density of separators, i.e.
	\begin{equation}\label{sep.density}
	p(t)=
	\lim_{n\to\infty}\frac{1}{2n}\sum_{v\in S_n}\mathbbm{1}_{\{\eta_v(t)\neq\eta_{v+1}(t)\}}\quad\text{a.s.}
	\end{equation}
    As a side note, we mention that $p(0)=2p_1p_2\in(0,\frac12]$, with $p(0)=1/2$ in the symmetric case.
	
Let $\epsilon>0$. In view of Lemma \ref{gentle} and the fact that $\Prob\big(N_v(t+\epsilon)-N_v(t)=0\big)=\mathrm{e}^{-\epsilon}$, using the
	union bound we can conclude that the probability for a fixed site $v$ not to be involved in a swap during $(t,t+\epsilon]$ is at least
	$\mathrm{e}^{-\epsilon}-\epsilon$.
	Consequently, it holds -- again by Lemma \ref{ergodic} -- that for any fixed $\epsilon\in(0,\frac12]$ a.s.\ there exist strictly increasing
	sequences $(U_n)_{n\in\N}$ and $(V_n)_{n\in\N}$ of random variables with $U_1,V_1>0$, such that no agent that occupies a site in
	$\{-U_n,V_n;\; n\in\N\}$ at time $t$ moves until time $t+\epsilon$.
	Hence, no separator can enter $S_n:=\{-U_n,\dots,V_n\}$ during this time interval (and no new separators can emerge by definition of
	the model). We conclude, by means of \eqref{sep.density}, that $p(t+\epsilon)\leq p(t)$, whence $p(t)$ is non-increasing with $t$.
\end{proof}

We will refer to the set of vertices between two consecutive separators as a \emph{monochromatic section}. Next, we establish a connection between the separators and the monochromatic sections of a configuration and investigate their evolution in time.

\begin{definition}\label{monchrint}
	For a fixed vertex $v\in \Z$, let $I_v(t)$ denote the monochromatic section containing $v$ at time $t$, i.e.\
	$I_v(t)=\{\underline{v}(t),\dots,v,\dots,\overline{v}(t)\}$, where
	\begin{align*}\underline{v}(t)&:=\min\{u\leq v;\;\eta_w(t)=\eta_v(t),\text{ for all }u\leq w\leq v\}\quad\quad\\
	\overline{v}(t)&:=\max\{u\geq v;\;\eta_w(t)=\eta_v(t),\text{ for all }v\leq w\leq u\}.
	\end{align*}
	Furthermore, for $l\in\N$, let $q_l(t):=\Prob\big(|I_v(t)|\leq l\big)$ be the probability that a fixed vertex is contained in a monochromatic
	section of length at most $l$ at time $t$.
\end{definition}

The following result shows that the possible asymptotic lengths of monochromatic sections depend crucially on the moving horizon $h$:
\begin{proposition}\label{shortvanish}
In the Schelling model on $\Z$ with two types, we have that \[\lim_{t\to\infty}q_l(t)=0,\quad\text{for all }l\leq h.\]
\end{proposition}

\begin{proof}
	Fix $l\leq h$ and assume for contradiction that $q_l(t)\nrightarrow0$ as $t\to\infty$. Then, by the non-negativity of $(q_l(t))_{t\geq 0}$, there exist $\delta>0$ and an increasing sequence of time points $(t_k)_{k\in\N}$ such that $t_{k+1}-t_k\geq 1$ and $q_l(t_k)\geq\delta$ for all $k\in\N$. Let $\Lambda_n:=\{1,\dots,n\}$ and $\mathring{\Lambda}_n:=\{2,\dots,n-1\}$, and write $A_n(t)$ for the event that at least two monochromatic sections are entirely contained in $\mathring{\Lambda}_n$ at time $t$, one of which has length at most $l$. The average
	\[X_n(t)=\frac{1}{n}\cdot\sum_{v=1}^{n}\mathbbm{1}_{\{|I_v(t)|\leq l\}}\]
	is a $[0,1]$-valued random variable with expectation $q_l(t)$. Hence
	\begin{equation}\label{Markov}
	q_l(t)=\E\big[X_n(t)\big]\leq \tfrac{3l}{n}+ \Prob\big(X_n(t)>\tfrac{3l}{n}\big)
	\end{equation}
	and, by our assumption, we can choose $N\in\N$ big enough such that $A_N(t_k)$ has probability at least $\frac{\delta}{2}$ uniformly in $k$ (as $A_N(t)$ is trivially fulfilled whenever $\sum_{v\in\Lambda_N}\mathbbm{1}_{\{|I_v(t)|\leq l\}}$ exceeds $3l$).
	Further, let $\epsilon=\frac{1}{2N}$ and write $B_{N}(t)$ for the event that no agent placed outside $\Lambda_N$ at time $t$ evaluates a site inside $\Lambda_N$ during the time interval $(t,t+\epsilon]$ for a potential move. Finally, conditioned on $A_N(t)$, let $C_N(t)$ be the following event: During the time period $(t,t+\epsilon]$, there are sufficiently many swaps inside $\mathring{\Lambda}_N$
	of an endpoint of a monochromatic section of length $\leq h$ with the neighboring site (occupied by an agent of different color) on the
	opposite side of the section (cf.\ Figure \ref{snail}) such that this section extends by means of bumping into a section of the same color
	(which is at least partly contained in $\Lambda_N$) and, in addition to that, no other swaps involving an agent placed inside
	$\Lambda_N$ occur during $(t,t+\epsilon]$.
	\begin{figure}[H]
	\centering
	\includegraphics[scale=0.86]{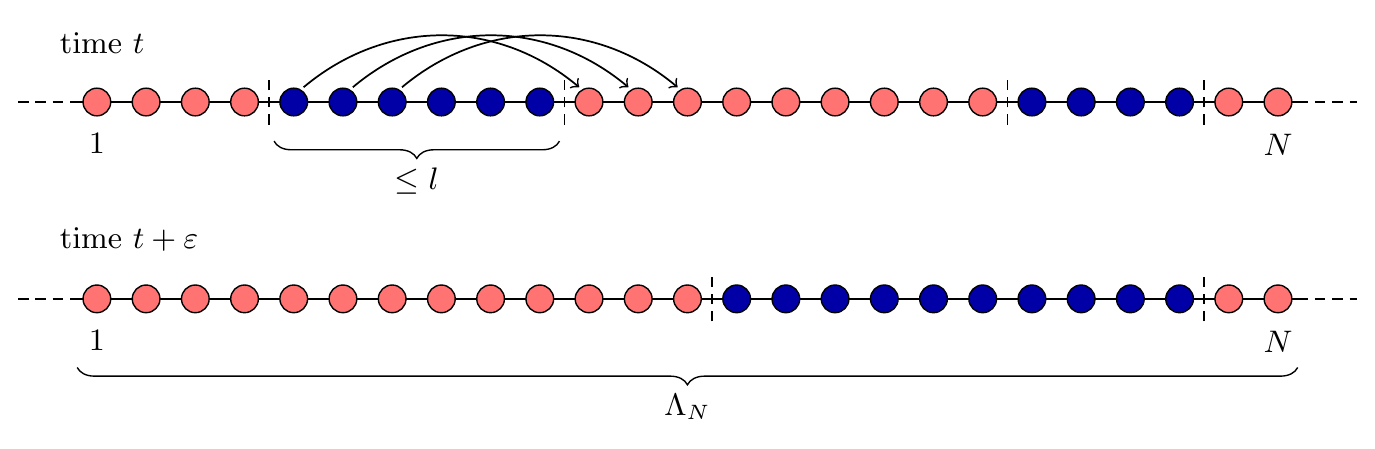}
	\caption{An illustration of the event $C_N(t)$: two monochromatic sections merge.\label{coalesce}}
\end{figure}

	We claim that $\Prob\big(A_N(t_k)\cap B_N(t_k)\cap C_N(t_k)\big)\geq p$ uniformly in $k$ for some $p>0$: To ease notation, let $R(v,t_1,t_2)$ be the randomness driving the dynamics at site $v$ during the time interval
	$(t_1,t_2]$, that is, $R(v,t_1,t_2)$ is the random entity given by $\{N_v(t);\;t_1<t\leq t_2\}$ and $\{Q_{v,n};\;N_v(t_1)<n\leq N_v(t_2)\}$. Since $B_N(t)$ is measurable with respect to $\{R(v,t,t+\epsilon);\;v\notin\Lambda_N\}$ and $A_N(t)\cap C_N(t)$ with respect to $\{R(v,0,t);\;v\in\Z\}\cup\{R(v,t,t+\epsilon);\;v\in\Lambda_N\}$, these two events are independent. From Corollary \ref{nopass} and the choice of $\epsilon$, it follows that $\Prob\big(B_N(t)\big)\geq\frac12$. Since $N=N(\delta,l)$ is deterministic and finite, conditioning on the finitely many possible configurations inside $\Lambda_N$ for which  $A_N(t)$ holds, we get (using \eqref{horizon} and the independence of $N_v, \boldsymbol{Q}_v$) a strictly
	positive lower bound for $\Prob(C_{N}(t))$, say $q>0$. In conclusion, it holds for all $k\in\N$ that
	$\Prob\big(A_N(t_k)\cap B_N(t_k)\cap C_N(t_k)\big)\geq \frac{\delta q}{4}$.
	
	Now partition $\Z$ into sections of length $N$, for instance $\{\Lambda_N+jN;\;j\in\Z\}$, and call a section {\em good} at time $t$, if the number of separators in the section decreases by at least $2$ during the time interval $(t,t+\epsilon]$, while the agents at its leftmost and rightmost site stay put. Observe that $A_N(t)\cap B_N(t)\cap C_N(t)$ guarantees that $\Lambda_N$ is a good at time $t$. Since no separator can
	jump over, enter or leave a good section, the number of separators on a line segment in between good sections is non-increasing.
	Lastly, by Lemma \ref{ergodic}, the density of good sections (a.s.) equals the probability of a section to be good. Let $e_v$ be short for the edge $\langle v,v+1\rangle$. Combining the above observations, we obtain that, for any $k\in \N$, almost surely
	\begin{align*}
	p(t_k)-p(t_k+\epsilon)&= \lim_{i,j\to\infty}\frac{1}{(i+j)N-1}\,\sum_{v=-iN+1}^{jN-1}
	\mathbbm{1}_{\{e_v\in \mathcal{S}(t_k)\}}-\mathbbm{1}_{\{e_v \in \mathcal{S}(t_k+\epsilon)\}}\\[-0.1cm]
	&\geq \frac{1}{N}\cdot\lim_{i,j\to\infty}\frac{1}{i+j}\,\sum_{m=-i}^{j-1}2\cdot\mathbbm{1}_{\{\Lambda_N+mN\text{ is good at time }t_k\}}\\
	&\geq\frac{\delta q}{2N}>0.
	\end{align*}
	Since this holds uniformly in $k$, in view of Lemma \ref{decrease} it contradicts the non-negativity of $p(t)$ and therefore disproves the assumption that $q_l(t)\nrightarrow0$ as $t\to\infty$.
	\vspace*{-1em}
\end{proof}

\section{Infinite moving horizon}\label{sec:unbounded}

In this section, we prove Theorem \ref{d=1,unbounded} where the asymptotics is settled for the two-color Schelling model in $d=1$ with $h=\infty$, that is, when arbitrarily long moves are possible. To this end, we first conclude from Proposition \ref{shortvanish} that, when $h=\infty$, the density of separators tends to 0 as $t\to\infty$.
\vspace*{1em}
\begin{lemma}\label{ptzero}
With $p(t)$ and $q_l(t)$ as defined in Section \ref{ss:sep}, the following are equivalent:\\[0.2cm]
$ (i) \lim\limits_{t\to\infty}p(t) =0 \qquad\qquad (ii) \lim\limits_{t\to\infty}q_l(t) =0, \text{ for all } l\in\N.$
\end{lemma}
\begin{proof} This equivalence readily follows from Lemma \ref{ergodic}, together with the relation between the density of sites in short monochromatic sections and the density of separators: On the one hand, in a given configuration, trivially there has to be a shift in color after at most $l$ sites that are included in monochromatic sections of length at most $l$. Hence, Lemma \ref{ergodic} directly implies that $p(t)\geq \frac{q_l(t)}{l}$. On the other hand, again by Lemma \ref{ergodic}, for $p(t)\geq \delta$ and $n$ sufficiently large the probability that  a section comprising $n$ sites at time $t$ contains at least $\frac{\delta n}{2}+2$ separators (and thus $\frac{\delta n}{2}$ monochromatic sections entirely) is larger than $\frac12$. Since at least half of the contained monochromatic sections must then be of length at most $\frac4\delta$, another application of Lemma \ref{ergodic} gives $q_{\lceil\frac{4}{\delta}\rceil}(t)\geq \frac{\delta}{4}$.
\end{proof}

Having established the above lemma, Theorem \ref{d=1,unbounded} follows without much further effort.
\vspace*{1em}

\begin{proof}[Proof of Theorem \ref{d=1,unbounded}]
	For $h=\infty$, Proposition \ref{shortvanish} and Lemma \ref{ptzero} together immediately imply that $\lim_{t\to\infty} p(t)=0$. This already proves the first claim, as by the union bound it trivially holds that
	\[\Prob\big(\eta_u(t)\neq\eta_v(t)\big)\leq |u-v|\cdot p(t).\]
	Pointwise convergence can easily be ruled out using the non-existence of local limits:
	Define $A_{v,1}$ to be the event that site $v$ eventually stays occupied by a blue agent -- that is, $A_{v,1}:=\{\lim_{t\to\infty}\eta_v(t)=1\}$ -- and
	assume for contradiction that $\Prob(A_{v,1})>0$. By Lemma \ref{ergodic}, we then have a positive density of eventually blue sites. Let $V$ be the first eventually blue site to the right of the origin and choose $n$ large enough to ensure that $\Prob(V\leq n)\geq p_1+\frac{p_2}{2}$. Then	
	 \[\Prob\big(\eta_0(t)=2\big)\leq\Prob\big(V>n\big)+\Prob\big(\eta_V(t)=2\big)+\sum_{v=0}^{n-1}\Prob\big(\eta_v(t)\neq\eta_{v+1}(t)\big).\]
        In view of \eqref{limit} and the fact that $V$ is eventually blue, all summands but the first one tend to 0 as $t\to\infty$. Thus,
	$\Prob\big(\eta_0(t)=2\big)<p_2$ for $t$ large enough, which contradicts Lemma \ref{probconst}. The same argument rules out that $v$ is eventually occupied by a red agent only. It follows that $\eta_v(t)$ a.s.\ never stops shifting between the values 1 and 2, forcing $|\eta_u(t)-\eta_v(t)|$ to take on the value $1$ at arbitrarily large times, as simultaneous swaps a.s.\ do not occur. This concludes the proof.
\end{proof}\vspace*{-0.5em}

\section{Finite moving horizon}\label{sec:bounded}

We now move on to the case when $h<\infty$ in the two-color model, and aim at proving Theorem \ref{d=1,bounded}. From Proposition \ref{shortvanish} we know that, in the unbounded case ($h=\infty$), the length of a monochromatic section grows continually in the sense that, for any $v\in\Z$, it holds that
\begin{equation}
\lim_{t\to\infty}\Prob\big(|I_v(t)|\leq l\big)=0\quad\text{for all }l\in\N.
\end{equation}
In the bounded case ($h\in\N$), the position of a monochromatic section of length strictly larger than $h$ is stable to the effect that there cannot be any swaps of agents placed on different ends/sides of the section. Monochromatic sections of length at most $h$ on the contrary can move one step at a time by means of swaps between one of their endpoints (either $\underline{v}(t)$ or $\overline{v}(t)$) with the neighboring agent of different color on the other end, see
Figure \ref{snail}.

\begin{figure}[H]
	\centering
	\includegraphics[scale=0.87]{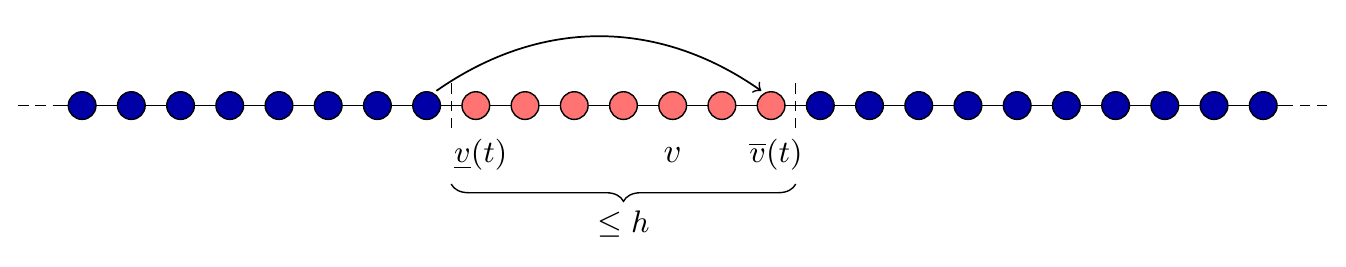}
	\caption{``Autonomous moving'' is only possible for monochromatic sections of length at most $h$.\label{snail}}
\end{figure}
 To capture the difference stemming from the length of the corresponding monochromatic section, we introduce two shades of the colors:
\begin{definition}\label{shades} Consider the two-color model on $\Z$.
	\begin{enumerate}[(a)]
		\item  A site $v\in\Z$ is called {\em light blue} at time $t$ if $\eta_v(t)=1$ and $|I_v(t)|\leq h$, and correspondingly {\em  dark blue} at time $t$ if $\eta_v(t)=1$ and $|I_v(t)|>h$. The shades {\em light red} and {\em dark red}
		are defined accordingly, with $\eta_v(t)=2$ instead.
		\item A site $v\in\Z$ is called {\em essentially blue} ({\em essentially red}) if there exists an almost surely finite time $T$,
		such that $v$ is not dark red (dark blue) at any time $t\geq T$.
	\end{enumerate}
\end{definition}

Note that the term ``eventually blue/red'' in the proof of Theorem \ref{d=1,unbounded} is stronger than the notion of ``essentially blue/red'' as it forbids both shades of the other color after an almost surely finite time. Observe further that, in Proposition \ref{shortvanish}, we already established that the light shades (i.e.\ the short monochromatic sections) gradually vanish as time evolves, even when $h\in\N$.

To prove Theorem \ref{d=1,bounded}, it will be useful to track the monochromatic sections dynamically, relating them to their current right endpoint, rather than relating them to fixed vertices, as in Definitions \ref{monchrint} and \ref{shades}.  At any given time $t$, the locations of separators (cf.\ \eqref{S(t)}) define the monochromatic sections in between them. We now relate these sections at different time points.
\begin{definition}\label{timeline}
Conditioned on the initial configuration, pick a monochromatic section, i.e.\ $I(0)=\{u,\dots,v\}$ such that
\[\eta_{u-1}(0)\neq\eta_u(0)=\eta_{u+1}(0)=\ldots=\eta_v(0)\neq\eta_{v+1}(0),\]
and let $\{I(t)\}_{0\leq t<T}$ be its {\em timeline} (in the sense of a chronological evolution), where $I(t)=\{U(t),\dots,V(t)\}\subseteq\Z$ is such
that for all $0\leq t<T$:
\begin{enumerate}[(i)]
	\item $I(t)$ is a monochromatic section, i.e.\ $\eta_{U(t)-1}(t)\neq\eta_{U(t)}(t)=\eta_{U(t)+1}(t)=\ldots=\eta_{V(t)}(t)\neq\eta_{V(t)+1}(t)$;
    \item $I(t)$ does not change color, i.e.\ $\lim_{s\nearrow t}\eta_{V(s)}(s)=\eta_{V(t)}(t)$;
    \item the right endpoint does not jump further than one site at a time, i.e.\ $|V(t)-\lim_{s\nearrow t}V(s)|\leq1$.
\end{enumerate}
Write $T\in[0,\infty]$ for the (random) first time at which the configuration does not feature any section that continues the timeline of $I(0)$ in the above
sense.
\end{definition}

Note that swaps typically move the endpoints of monochromatic sections merely to neighboring vertices, however, when a singleton agent swaps and two monochromatic sections of the same color coalesce, the timelines of both the singleton and the monochromatic section to its left end, while the merged one is the continuation of the monochromatic section to its right, see Figure \ref{monochrom} for an illustration. Furthermore, observe that in the timeline $\{I(t)\}_{0\leq t<T}$ of a fixed monochromatic section $I(0)$ in the initial configuration, any vertex $w\in\Z$ can change between $w\in I(t)$ and $w\notin I(t)$ an arbitrary number of times.
	\begin{figure}[H]
	\centering
	\includegraphics[scale=0.87]{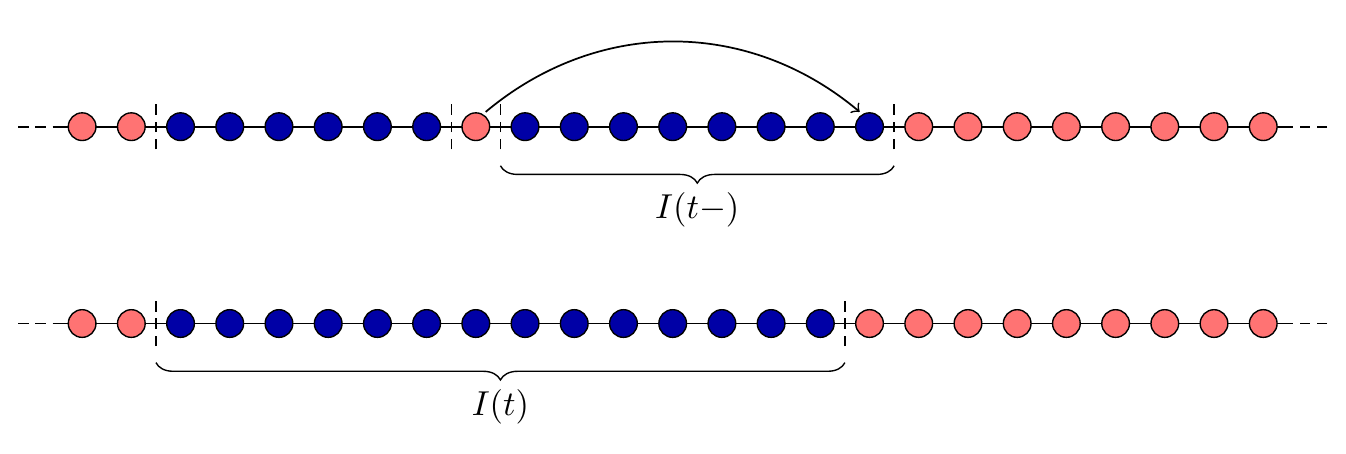}
	\caption{Coalescence ends the timeline of two monochromatic sections.\label{monochrom}}
\end{figure}

While Proposition \ref{shortvanish} showed that the probability of a site to be either light blue or light red tends to $0$, we are
now going to show that (a.s.) no vertex can change between the two dark shades infinitely often.

\begin{lemma}\label{allessential}
	For any $v\in \Z$, we have that $\Prob(v\textup{ is essentially blue})+\Prob(v\textup{ is essentially red})=1$ and the two events are disjoint.
	\end{lemma}

\begin{proof}
	In order to define a suitable mass transport based on the timeline of a monochromatic section, we first describe three crucial events in such a chronological evolution:
\begin{enumerate}[(i)]
	\item the section becomes mobile, i.e.\ $|I(t-)|=h+1$, $|I(t)|=h$;
	\item the section becomes immobile, i.e.\ $|I(t-)|=h$, $|I(t)|=h+1$;
	\item the section vanishes, i.e.\ $I(T-)=\{w\}$ and the agent at
	$w$ is involved in a swap at time $T$ (which ends the timeline of $I(0)$).
\end{enumerate}
	The mass transport $m$ is defined as follows: If site $v$ is the rightmost site of a monochromatic section $I(0)$ in the initial configuration, it sends mass $1$ equally distributed to the sites of $I(t)$ for each time $0\leq t<T$, at which either (i) or (ii) occurs. If additionally (iii) occurs, the site $v$ sends two times mass $1$, equally distributed to the nodes of $I_l\cup\{w\}$ and $I_r\cup\{w\}$ respectively, where $I_l$, $I_r$ are the monochromatic sections of opposite color left and right of $w$,
	which coalesce at time $T$. No other mass is sent. In particular, no mass is sent from $v$ if $\eta_v(0)=\eta_{v+1}(0)$.

It is crucial to observe that, in scenario (i), the interval $I(t)$ must be neighbored by a mobile monochromatic section (i.e.\ of length $\leq h$).
Whenever two monochromatic sections of length $\leq h$ are next to each other, there is a non-zero probability that one ends the timeline of the other by traversing it, cf.\ Figure \ref{coalesce}. This probability can be bounded from below
(uniformly in the lengths of the two sections in question and the surrounding configuration) by some $r>0$. Consequently, the number of times (i) can happen is stochastically dominated by a geometric random variable (on $\N$) with parameter $r$. The probability that $v$ is the rightmost site of some $I(0)$ equals $2p_1p_2$.
Since (i) and (ii) must alternate and (iii) can happen at most once, we find that
\begin{equation}\label{finitemass}
\E\big[m(v,\Z)\big]\leq 2p_1p_2\cdot\big(\tfrac2r+2\big),
\end{equation}
where the extra $+2$ accounts for the additional mass sent if (iii) occurs.

Now fix $v\in\Z$ and note that an immediate change of site $v$ from dark blue to dark red (or the reverse) is impossible. Instead, to change from one dark shade to the other, one of the following two scenarios has to occur:
\begin{itemize}
\item Case 1: The length $|I_v(t)|$ of the monochromatic section currently including $v$, changes from $>h$ to $\leq h$ and once mobile, the section either moves such that it no longer includes $v$ or ceases to exist in the sense of (iii), with $v$ in place of $w$.
\item Case 2: A mobile monochromatic section of opposite color moves such that it includes $v$ and then extends to length
$>h$ by means of either (ii) or a coalescence caused by a vanishing neighboring section, as described in (iii).
\end{itemize}
As a consequence, every such change causes $v$ to receive at least mass $\frac{1}{h+1}$. By Lemma \ref{MTP} and \eqref{finitemass} this can happen only finitely many times. Finally, Proposition \ref{shortvanish} implies that
\begin{align*}\Prob(v\textup{ is both essentially blue and essentially red})
						&=\Prob\Big(\liminf_{t\to\infty}\big\{|I_v(t)|\leq h\big\}\Big)\\
						&\leq \lim_{t\to\infty}q_h(t)=0,
\end{align*}
which concludes the proof.
\end{proof}

Next, we show that, if an essentially red site is a neighbor of an essentially blue site, then the occupation of both sites has to fixate. Crucial in the argumentation will be the fact that, for any short monochromatic section (i.e.\
$I(t)=\{u,\dots,v\}$ with $|I(t)|\leq h$), with probability bounded away from 0 (uniformly in its color, length, $t$ and the
surrounding configuration) the next change in its chronological evolution is a move of the whole section to the right, by means
of a swap of the agents at sites $u$ and $v+1$ (likewise to the left: agents at $u-1$ and $v$ swap, cf.\ Figure \ref{snail}).

\begin{lemma}\label{essentailneighbors}
	An essentially blue site $v$ and an essentially red site $u$ can only be neighbors if both
	$\lim_{t\to\infty}\eta_v(t)$ and $\lim_{t\to\infty}\eta_u(t)$ exist (and then equal $1$ and $2$ respectively).
\end{lemma}

\begin{proof}
	Fix $u\in\{v-1,v+1\}$. Let $A$ be the event that $v$ is essentially blue, $u$ is essentially red and at least one of $\lim_{t\to\infty}\eta_u(t)$ and $\lim_{t\to\infty}\eta_v(t)$ does not exist. Assume for contradiction that $\Prob(A)=s>0$. For $t\in[0,\infty)$, define $A(t)$ to be the event that neither $v$ is dark red nor $u$ is dark blue at any time $t'\geq t$. By Definition \ref{shades} and Proposition \ref{shortvanish}, we can choose $t_0$ large enough, such that $\Prob(A(t_0)\cap A)\geq\frac{s}{2}$ and $q(t)\leq\frac{s}{5}$ for all $t\geq t_0$. Consequently, conditioned on $A(t_0)\cap A$, for each $t\geq t_0$, the site $v$ will be dark blue and $u$ dark red at time $t$ with probability at least $\frac{1}{10}$.
	
	Given that at least one of $\lim_{t\to\infty}\eta_u(t)$ and $\lim_{t\to\infty}\eta_v(t)$ does not exist and $u$ and $v$ are dark blue and dark red respectively at time $t$, there must be a first time $\xi>t$, at which exactly one of $u$ and $v$, say $u$, changes to the light shade of its current color. By the fact mentioned just before the lemma, there is a probability bounded away from 0 that the next change in the chronological evolution of the short red monochromatic section $I_u(\xi)$ is a move away from $v$, which turns $u$ dark blue. By the conditional Borel-Cantelli
	Lemma (see for instance \cite[Corollary 6.20]{Foundations}), conditioned on $A(t_0)\cap A$ this will happen almost surely for
	some $t\geq t_0$, contradicting the definition of $A(t_0)$. The same argument can be applied if instead $v$ turns light blue before $u$ turns light red after time $t$.
\end{proof}

With these lemmas in hand, we are all set for proving Theorem \ref{d=1,bounded}. The final step is to show that, with probability 1, any given site is not
only either {\em essentially} blue or red, but also {\em eventually} blue or red. In other words, what has to be ruled out is the persistent inclusion of an essentially blue (red) site in short red (blue) monochromatic sections.
\vspace*{1em}

\begin{proof}[Proof of Theorem \ref{d=1,bounded}]
Let $B_v$ denote the event that site $v$ is essentially blue and $\lim_{t\to\infty}\eta_v(t)$ does not exist. Conditioned on
$B_v$, the sites $v-1$ and $v+1$ also have to be (a.s.) essentially blue by Lemmas \ref{allessential} and \ref{essentailneighbors}. Further, if $v$ is
part of a red monochromatic section of length $\leq h$ at arbitrarily large times, this (a.s.)\ holds for its neighboring sites as
well (by the fact stated just before Lemma \ref{essentailneighbors} and the conditional Borel-Cantelli Lemma). Let $B:=\bigcap_{u\in\Z} B_u$. Using induction, we arrive at
\[\Prob(B)=\lim_{n\to\infty}\Prob\bigg(\bigcap_{u=v-n}^{v+n}B_u\bigg)=\Prob(B_v).\] Note that $B$ is a translation invariant event and must hence have probability either 0 or 1 by ergodicity (Lemma \ref{ergodic}). The same argument applies to $R:=\bigcap_{u\in\Z} R_u$, where $R_v$ denotes the event that $v$ is essentially red
and $\lim_{t\to\infty}\eta_v(t)$ does not exist. We conclude that, under the assumption that with positive probability $\lim_{t\to\infty}\eta_v(t)$ does not exist, Lemma \ref{allessential} implies that either $B_v$ or $R_v$ must have probability 1. This, however, is impossible: If $v$ is almost surely essentially blue, it follows from Proposition \ref{shortvanish} that
\[\Prob\big(\eta_v(t)=1\big)\geq 1-\Prob(v\text{ is dark red at time }t)-q(t)>p_1\]
for $t$ large enough, contradicting Lemma \ref{probconst}. Similarly, $v$ cannot be a.s.\ essentially red, which disproves our assumption and with that concludes the proof.
\end{proof}

\section{Lazy and easily pleased agents}\label{lazy}

In this short section, we prove Propositions \ref{prop:lazy} and \ref{prop:tau<0}, stating that the non-fixation in the unbounded case is eliminated when agents are lazy (that is, when agents only move when they strictly improve their situation) and when the satisfaction bar is lowered, respectively. Hence, for any $h\in \N\cup\{\infty\}$, the state $\eta_v(t)$ converges almost surely for all $v\in\Z$.

\begin{proof}[Proof of Proposition \ref{prop:lazy}]
As pointed out in Section \ref{ss:res}, when the agents are lazy, every swap involves at least one {\em singleton agent}, i.e.\ an agent currently having no neighbor of its own color. The key to verify Proposition \ref{prop:lazy} is to track the occurrence, or rather disappearance, of singletons in the configuration. To this end, let us define the following mass transport, based on the timelines of the monochromatic sections: Site $u$ sends mass $1$ to sites $v$ and $w$ respectively, if $u$ is the rightmost site of a monochromatic section $I(0)$ in the initial configuration, the timeline of which ends by means of a switch of the agents at $v$ and $w$ at some time $T$, when $I(T-)=\{v\}$. As every swap includes at least one singleton agent joining a monochromatic section of her color, from Lemma \ref{MTP} it follows
that $2$ is an upper bound on the expected number of swaps at a given site. Consequently, the number of times at which the color
at vertex $v$ changes has to be a.s.\ finite and this implies the claim.
\end{proof}

By the same token, we get the analogous result for the model with satisfaction threshold $\tau<0$ (irrespectively of whether the agents are lazy or not), since in this regime only agents with {\em no} neighbor of their own color will be dissatisfied.

\section{Multiple types}\label{sec:ex}

Now let us go back to the original model (with $\tau=0$ and non-lazy agents) but consider the case with more than two types of agents, i.e.\ $c\geq 3$. Introducing more than two types of agents entails a qualitative change in the dynamics: For $c\geq 3$, it is possible that singleton agents (i.e.\ agents with no neighbor of their own kind) stay singletons after a swap (namely when maximally dissatisfied agents move to equally bad places as illustrated in Figure \ref{jumping_separator}); we will refer to this as a {\em singleton jump}.
As a consequence, separators (being edges with unlike incident neighbors as before) can now move more than one step at a time. Observe that this kind of swap was impossible in the two type case ($c=2$), as it would have involved an agent with two alike neighbors, which is unwilling to move. In addition to that, for $c\geq 3$, two separators that coalesce can either both vanish or collapse into one, depending on the combination of incident colors.

\begin{figure}[H]
\centering
\includegraphics[scale=0.88]{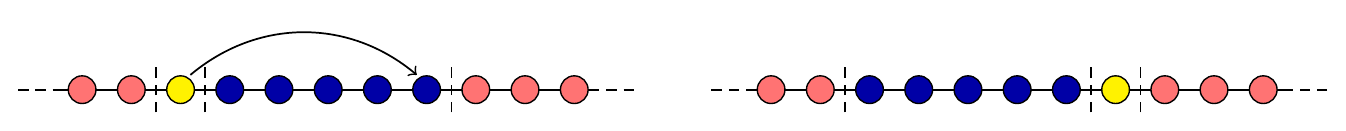}
\caption{With more than two types of agents, separators do not move in simple random walks anymore.\label{jumping_separator}}
\end{figure}



\subsection{Unbounded moving horizon}

Despite the qualitative changes described above, the result of Theorem \ref{d=1,unbounded}, valid when $h=\infty$, extends to the case with more than two types of agents.
\begin{theorem}\label{c>2}
	Theorem \ref{d=1,unbounded} still holds with $c\geq 3$ agent types.
\end{theorem}

\noindent In order to prove this generalization, we will reuse and slightly modify parts of the proof of Proposition \ref{shortvanish} and then conclude
as in Theorem \ref{d=1,unbounded}. First we extend Lemma \ref{decrease} to the case with $c\geq 3$ types by means of a minor adjustment of its proof.
\begin{lemma}\label{multidecrease}
Consider the Schelling model on $\Z$ with $c\geq 3$ types. Then, $t\mapsto p(t)$ is non-increasing in $t\geq0$.
\end{lemma}

\begin{proof}
As in the proof of Lemma \ref{decrease}, fix $\epsilon>0$ and choose strictly increasing sequences $(U_n)_{n\in\N}$ and $(V_n)_{n\in\N}$
of random variables (with $U_1,V_1>0$), such that all agents at a site in $\{-U_n,V_n;\; n\in\N\}$ stay put during $(t,t+\epsilon]$. Further,
let $m(u,v)$ count the number of times a singleton jump occurs from site $u$ to $v$ during this time period.
As $\E\big[ m(v,\Z)\big]\leq\epsilon$, Lemmas \ref{ergodic} and \ref{MTP} yield that almost surely
\begin{align*}
p(t)&=\lim_{n\to\infty}\frac{1}{U_n+V_n}\sum_{v=-U_n}^{V_n-1}\mathbbm{1}_{\{\eta_v(t)\neq\eta_{v+1}(t)\}}\\
	  &\geq\lim_{n\to\infty}\frac{1}{U_n+V_n}\sum_{v=-U_n}^{V_n-1}\mathbbm{1}_{\{\eta_v(t+\epsilon)\neq\eta_{v+1}(t+\epsilon)\}}+m(\Z,v)-m(v,\Z)\\
	  &=p(t+\epsilon)+\E \big[m(\Z,0)-m(0,\Z)\big] = p(t+\epsilon).
\end{align*}\vspace*{-3	em}

\end{proof}

With this in place, we can prove Theorem \ref{c>2}.\vspace{1em}

\begin{proof}[Proof of Theorem \ref{c>2}]
	Roughly following the proof of Proposition \ref{shortvanish}, assume for contradiction that $\lim_{t\to\infty}p(t)=p>0$ and let $A_n(t)$
	be the event that the number of separators in $\mathring{\Lambda}_n:=\{2,\dots,n-1\}$ at time $t$ is at least $\binom{c}{2}+3$; this
	ensures that there are two separators, incident to the same two colors, and that the monochromatic sections that they are incident to are
	entirely contained in $\mathring{\Lambda}_n$. With help of the random variable
	\[X_n(t)=\frac{1}{n-2}\cdot\sum_{v=2}^{n-1}\mathbbm{1}_{\{\eta_v(t)\neq\eta_{v+1}(t)\}},\]
	similar to \eqref{Markov}, we can deduce that $\Prob\big(A_n(t)\big)\geq p-\frac{\binom{c}{2}+2}{n-2}$ and take $N$ big enough to ensure that
	$\Prob\big(A_N(t)\big)\geq \frac{p}{2}$.
	Choose $\epsilon=\frac{1}{2N}$ and $B_N(t)$ as in the proof of Proposition \ref{shortvanish}. Given $A_N(t)$, pick one such pair of
	separators (incident to say blue and yellow) and label them $\alpha$ and $\beta$. Depending on whether they are
	neighbors or not, incident to these two, we find in total either 3 or 4 monochromatic sections. Let $I_1(t)$ be the yellow
	section incident to $\alpha$ and $I_2(t)$ be the blue section incident to $\beta$ (cf.\ Figure \ref{multiculti}).
   	
	Given $A_N(t)$, we define $C_N(t)$ to be the following event: All swaps involving an agent inside $\Lambda_N$ during the time period
	$(t,t+\epsilon]$ occur between the site in $I_1$ currently incident to $\alpha$ and the site in $I_2$ currently incident to $\beta$ and,
	further, there are exactly $\min\{|I_1(t)|,|I_2(t)|\}$ such swaps -- ending the timeline of at least one of $I_1(t)$ and $I_2(t)$. By the same line of
	reasoning as in the proof of Proposition \ref{shortvanish}, the event $A_N(t)\cap B_N(t)\cap C_N(t)$ has positive probability, bounded from below
	by some $q>0$ (uniformly in $t$), and ensures that the section $\Lambda_N$ is {\em good} at time $t$, which here means that the number of
	separators it contains decreases by at least $1$ during the time interval $(t,t+\epsilon]$, while the agents at its leftmost and rightmost site stay put.

    \begin{figure}
   	\centering
   	\includegraphics[scale=0.86]{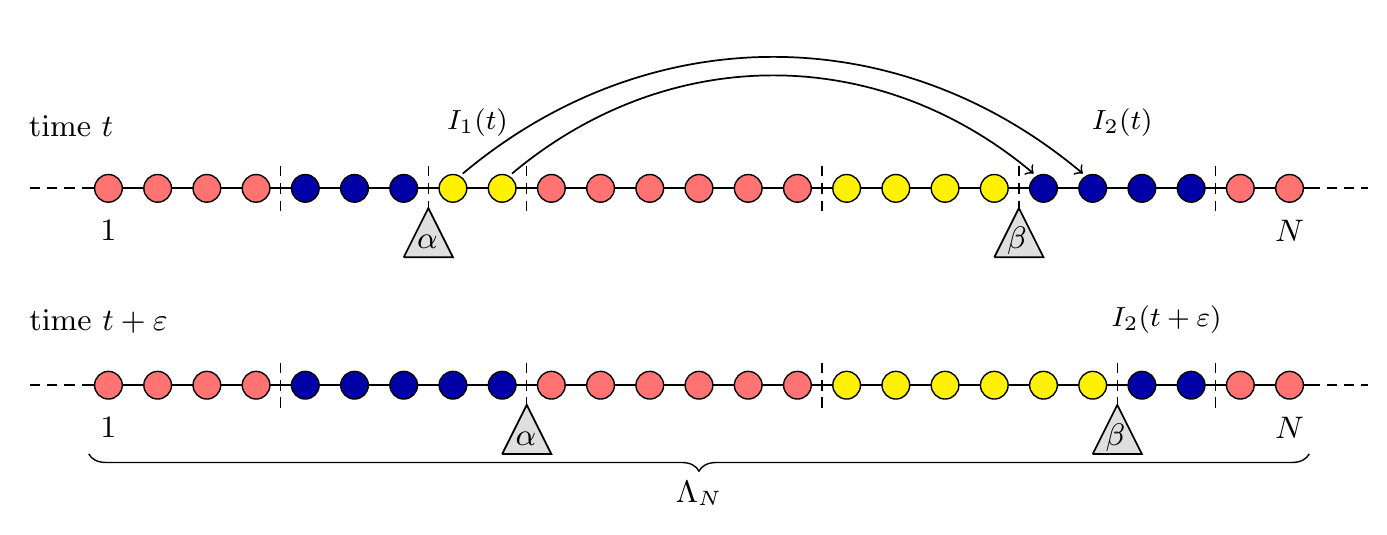}
   	\caption{Forcing separators to coalesce (by means of local modification) works with more than two kinds of agents as well.\label{multiculti}}
   \end{figure}

    With $m(u,v)$ as in the proof of Lemma \ref{multidecrease}, the partition $\Z=\bigcup_{j\in\Z}\Lambda_N+jN$ and another application of Lemmas
    \ref{ergodic} and \ref{MTP} yield
    \[\E \big[m(\Z,0)-m(0,\Z)\big]+\lim_{i,j\to\infty}\frac{1}{N(i+j)}\,\sum_{m=-i}^{j-1}\mathbbm{1}_{\{\Lambda_N+mN\text{ good at time }t\}}
	\geq\frac{q}{N}\]
	as a positive lower bound for $p(t)-p(t+\epsilon)$, uniformly in $t$. Hence, we have arrived at the same contradiction as in the proof of Proposition \ref{shortvanish} and can conclude, as in the proof of Lemma \ref{ptzero}, that $\lim_{t\to\infty}p(t)=0$. The proof of Theorem \ref{d=1,unbounded} then applies verbatim.\vspace*{-1em}
\end{proof}
\smallskip
\subsection{Bounded moving horizon}

For $h \in\mathbb{N}$, Theorem 1.2 states that the color at a given site converges almost surely in the model with two types of agents. With more than two types, there are configurations where the color of certain sites does not converge on a torus. Specifically, for $c=3$, a yellow singleton next to a red interval of length exactly $h$, embedded on both sides of blue intervals of length strictly larger than $h$, could potentially keep on jumping back and forth over the red interval forever; see Figure \ref{fig:forever}. With $c>3$, also more involved scenarios where a singleton agent jumps between a finite number of sites are possible. We do not know whether or not these semi-fixated sections will occur on $\Z$, since indeed it requires that the local configuration is not disturbed by changes far away. The answer may depend on the proportions of the types. Apart from scenarios where singletons locally jump back and forth, we conjecture that the configuration freezes also for $c\geq 3$ when $h\in\mathbb{N}$. We leave a detailed analysis of this case for future work.

 \begin{figure}[H]
   	\centering
   	\includegraphics[scale=0.86]{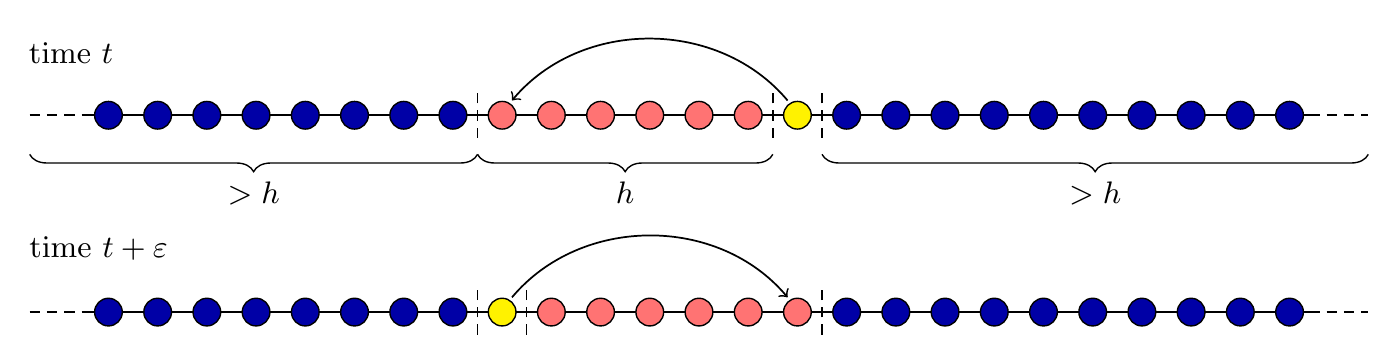}
   	\caption{A configuration with $c=3$ where certain sites may not converge in a finite system due to repeated singleton jumps. \label{fig:forever}}
   \end{figure}

\end{document}